\documentclass[11pt, twoside, fullpage]{article}
\usepackage{amssymb, amsmath, amsthm, verbatim}
\usepackage[top=30mm, left=23mm, right=23mm, bottom=30mm]{geometry}
\usepackage{inputenc}
\usepackage{fancyhdr}
\usepackage{tikz}
\usepackage{verbatim}
\usetikzlibrary{positioning}
\usepackage{titling}
\usepackage[hyphens]{url}
\usepackage{hyperref}
\usepackage[hyphenbreaks]{breakurl}
\usepackage[T1]{fontenc}
\usepackage{currvita}
\usepackage{xcolor}
\predate{}
\postdate{}
\usepackage{lipsum}
\newtheorem{theorem}{Theorem}

\newtheorem{lemma}[theorem]{Lemma}

\newtheorem{conjecture}[theorem]{Conjecture}

\newcommand{\D}{\mathcal{D}}

\newcommand{\ex}{\mathrm{ex}}
\newcommand{\exc}{\mathrm{exc}}
\begin{document}
\newcommand{\Addresses}{{
\bigskip
\footnotesize
\medskip

\noindent David~Ellis, \textsc{School of Mathematics, University of Bristol, Bristol, UK.}\par\noindent\nopagebreak\textit{Email address: }\texttt{david.ellis@bristol.ac.uk}

\medskip 

\noindent Maria-Romina~Ivan, \textsc{Department of Pure Mathematics and Mathematical Statistics, Centre for Mathematical Sciences, Wilberforce Road, Cambridge, CB3 0WB, UK.}\par\noindent\nopagebreak\textit{Email address: }\texttt{mri25@dpmms.cam.ac.uk}

\medskip

\noindent Imre~Leader, \textsc{Department of Pure Mathematics and Mathematical Statistics, Centre for Mathematical Sciences, Wilberforce Road, Cambridge, CB3 0WB, UK.}\par\noindent\nopagebreak\textit{Email address: }\texttt{i.leader@dpmms.cam.ac.uk}}}

\pagestyle{fancy}
\fancyhf{}
\fancyhead [LE, RO] {\thepage}
\fancyhead [CE] {DAVID ELLIS, MARIA-ROMINA IVAN AND IMRE LEADER}
\fancyhead [CO] {TUR\'AN DENSITIES FOR DAISIES AND HYPERCUBES}
\renewcommand{\headrulewidth}{0pt}
\renewcommand{\l}{\rule{6em}{1pt}\ }
\title{\Large{\textbf{TUR\'AN DENSITIES FOR DAISIES AND HYPERCUBES}}}
\author{DAVID ELLIS, MARIA-ROMINA IVAN AND IMRE LEADER}
\date{}
\maketitle
\begin{abstract}
An $r$-daisy is an $r$-uniform hypergraph consisting of the six $r$-sets formed by taking the union of an $(r-2)$-set with
each of the 2-sets of a disjoint 4-set. Bollob\'as, Leader and Malvenuto, and also Bukh, conjectured
that the Tur\'an density of the $r$-daisy tends to zero as $r \to \infty$. In this
paper we disprove this conjecture.

Adapting our construction, we are also able to disprove a folklore conjecture about
Tur\'an densities of hypercubes. For fixed $d$ and large $n$, we show that the smallest set of vertices of the
$n$-dimensional hypercube $Q_n$ that intersects every copy of $Q_d$ has asymptotic density strictly
below $1/(d+1)$, for all $d \geq 8$. In fact, we show that this asymptotic density is at most $c^d$, for some constant $c<1$. 
As a consequence, we obtain similar bounds for the edge-Tur\'an densities of hypercubes.
We also answer some related questions of
Johnson and Talbot, and disprove a conjecture made by Bukh and by Griggs and Lu on poset densities.

({\em MSC class:} 05D05.)
\end{abstract}
\section{Introduction}
As usual, if $X$ is a set, we write $X^{(r)}$ for the set of all $r$-element subsets of $X$. For an integer $r \geq 2$, an {\em $r$-daisy} is an $r$-uniform hypergraph of the form
$$\{S \cup X:\ X \in T^{(2)}\},$$
where $S$ is an $(r-2)$-element set and $T$ is a 4-element set disjoint from $S$. We denote this hypergraph by $\mathcal{D}_r$; note that $\mathcal{D}_r$ has $r+2$ vertices and six edges. We call the set $S$ the {\em stem} of the daisy. For any $r$-uniform hypergraph $\mathcal{H}$, we write $\pi(\mathcal{H})$ for the {\em Tur\'an density} of $\mathcal{H}$, i.e.\ $\pi(\mathcal{H})= \lim_{n \to \infty} (\ex(n,\mathcal{H})/{n \choose r})$, where $$\ex(n,\mathcal{H})= \max\{|\mathcal{F}|:\ \mathcal{F} \text{ is an } \mathcal{H}\text{-free, }r\text{-uniform hypergraph on }n\text{ vertices}\}.$$
Bollob\'as, Leader and Malvenuto made the following conjecture in \cite{blm}. (Bukh also made the same conjecture independently -- see \cite{blm}.)
\begin{conjecture}
\label{conj:daisies}
    $\pi(\mathcal{D}_r) \to 0$ as $r \to \infty$.
\end{conjecture}
We note that an easy averaging argument, averaging over links of vertices, shows that the sequence $(\pi(\mathcal{D}_r))_{r \geq 2}$ is monotone decreasing, so $\pi(\mathcal{D}_r)$ does indeed tend to some limit as $r \to \infty$.
\newline

More generally, for $s,t \in \mathbb{N}$ with $t \geq s$, we define an {\em $(r,s,t)$-daisy} to be an $r$-uniform hypergraph of the form
$$\{S \cup X:\ X \in T^{(s)}\},$$
where $S$ is an $(r-s)$-element set and $T$ is a $t$-element set disjoint from $S$. We denote this hypergraph by $\mathcal{D}_r(s,t)$. (As before, we call $S$ the {\em stem} of the daisy.) Bollob\'as, Leader and Malvenuto \cite{blm} in fact made the following conjecture, generalising Conjecture \ref{conj:daisies}.
\begin{conjecture}
\label{conj:daisies-gen}
    For any $s,t \in \mathbb{N}$ with $t \geq s$, we have $\pi(\mathcal{D}_r(s,t)) \to 0$ as $r \to \infty$.
\end{conjecture}

Conjecture \ref{conj:daisies} has the following appealing reformulation in terms of subcubes. As usual, a {\em $d$-dimensional subcube} of $\{0,1\}^n$ is a subset of $\{0,1\}^n$ of the form $\{x \in \{0,1\}^n:\ x_{i} = c_{i} \text{ for all } i \in F\}$, where $F \in \{1,2,\ldots,n\}^{(n-d)}$ and $c_i \in \{0,1\}$ for all $i \in F$. If $n \geq 4$ is even, a {\em middle 4-cube} of $\{0,1\}^n$ is a 4-dimensional subcube of $\{0,1\}^n$ which is contained in the middle five layers of $\{0,1\}^n$ --- meaning, those layers of Hamming weight between $n/2-2$ and $n/2+2$. Note that the intersection of a middle 4-cube with the middle layer of $\{0,1\}^n$ is precisely an $(n/2)$-daisy (identifying $\{0,1\}^n$ with the power-set of $\{1,2,\ldots,n\}$, in the usual way). Similarly, for each even integer $d$ and each even $n \geq d$, a {\em middle $d$-cube} of $\{0,1\}^n$ is a $d$-dimensional subcube of $\{0,1\}^n$ which is contained within the middle $d+1$ layers of $\{0,1\}^n$. As observed in \cite{blm}, Conjecture \ref{conj:daisies} is equivalent to the following.
\begin{conjecture}
\label{conj:rephr}
    If $n$ is even and $\mathcal{F} \subset \{x \in \{0,1\}^n:\ \sum_{i=1}^{n}x_i=n/2\}$ does not contain the middle layer of a middle 4-cube, then $|\mathcal{F}|=o(\binom{n}{n/2})$.
\end{conjecture}
Or, to rephrase in the language of Boolean functions, Conjecture \ref{conj:rephr} says that for any $\delta>0$, a Boolean function on the middle layer of $\{0,1\}^n$ with expectation at least $\delta$ must be identically 1 on (the middle layer of) some middle 4-cube, provided $n$ is sufficiently large depending on $\delta$. Similarly, Conjecture \ref{conj:daisies-gen} is equivalent to the statement that for any $\delta>0$ and even $d \geq 4$, a Boolean function on the middle layer of $\{0,1\}^n$ with expectation at least $\delta$ must be constant on (the middle layer of) some middle $d$-cube, provided $n$ is sufficiently large depending on $\delta$ and $d$. Unsurprisingly, these conjectures have generated much interest in the field of `analysis of Boolean functions', since they say, very roughly, that there are no analogues of parity-type functions on $\{0,1\}^n$ for Boolean functions defined on the middle layer alone.
\vspace{0.5cm}

Rather unexpectedly, Conjecture \ref{conj:daisies} is false. Our first aim in this paper is to disprove it, as follows.
\begin{theorem}
\label{thm:main} We have
    $$\lim_{r \to \infty}\pi(\mathcal{D}_r) \geq \prod_{k=1}^{\infty}(1-2^{-k}) \approx 0.29.$$
\end{theorem}

Previously, the best known lower bound on $\pi(\mathcal{D}_r)$ (for general $r$) was $\Omega(1/r)$, due to Ellis and King \cite{ek}. 
\vspace{0.5cm}

For $(r,2,t)$-daisies, we obtain the following lower bound.
\begin{theorem}  
For each $t\geq4$ we have 
$$\lim_{r \to \infty}\pi(\mathcal{D}_r(2,t)) \geq 1-1/t-o(1/t).$$
\label{thm:larger}
\end{theorem}
%

The bound in this theorem is clearly best-possible, up to the error term. Indeed, by averaging, if $n \geq r+t-2$ then for any $r$-uniform hypergraph $\mathcal{F}$ on the vertex-set $[n]=\{1,2,\ldots,n\}$ with density more than $1-1/(t-1)+1/(n-r+1)$, there exists an $(r-2)$-set $S$ such that the link graph $$\{\{i,j\} \in ([n]\setminus S)^{(2)}:\ S \cup \{i,j\} \in \mathcal{F}\}$$ has density at least $1-1/(t-1)+1/(n-r+1)$. So by Tur\'an's theorem this graph must contain a $K_t$, and so $\mathcal{F}$ must contain a copy of $\D_r(2,t)$. This shows that $\pi(\D_r(2,t)) \leq 1-1/(t-1)$ for all $r \geq 2$ and all $t \geq 4$.\\

For $(r,s,t)$-daisies with larger $s$, we obtain even stronger bounds, although these are probably not asymptotically sharp. It is these bounds that will be key to proving our applications for Tur\'an densities of hypercubes. The precise result we will need is the following. 

\begin{theorem}
\label{theorem:special}
    Let $k$ be even. Then we have    
    $$\lim_{r \to \infty} \pi(\mathcal{D}_{r}(k,8k+1)) = 1-O(2^{-k}).$$


\end{theorem}

We now discuss applications of the above results to Tur\'an problems in the hypercube. The latter were popularized by Alon, Krech and Szab\'o in \cite{aks}, following earlier results of Kostochka \cite{kost} and of Johnson and Entringer \cite{je}; Johnson and Talbot obtained further results and posed further problems in \cite{jt}. Our terminology follows that of Johnson and Talbot. As usual, the {\em $n$-dimensional hypercube graph} $Q_n$ is the graph with vertex-set $\{0,1\}^n$, where two zero-one vectors are joined by an edge if they differ in just one coordinate. We start with vertex-Tur\'an problems.

For $n \in \mathbb{N}$, we define 
$$\exc(n,Q_d)= \max\{|\mathcal{F}|:\ \mathcal{F} \subset \{0,1\}^n,\ \mathcal{F} \text{ is }Q_d\text{-free}\},$$
i.e.\ $\exc(n,Q_d)$ is the maximum possible size of a set of vertices in $Q_n$ that does not contain the vertex-set of a $d$-dimensional subcube. We define the {\em vertex-Tur\'an density of $Q_d$} to be
$$\lambda(Q_d) = \lim_{n \to \infty}\frac{\exc(n,Q_d)}{2^n}.$$
(An easy averaging argument shows that the above sequence of quotients is monotone decreasing, so the above limit exists.)

A by now well-known problem in extremal combinatorics is to determine $\lambda(Q_d)$ for each integer $d$. Kostochka \cite{kost}, and independently Johnson and Entringer \cite{je}, showed that $\lambda(Q_2)=2/3$, and in fact both sets of authors determined $\exc(n,Q_2)$ exactly for all $n \in \mathbb{N}$. For each 
$d \geq 3$, however, the value of $\lambda(Q_d)$ is unknown. 

It is in fact more convenient to consider the (equivalent) `complementary' problem: what is the minimal size $g(n,d)$ of a subset of $\{0,1\}^n$ that intersects the vertex-set of every $d$-dimensional subcube? Writing
$$\gamma_d = \lim_{n \to \infty} g(n,d)/2^n,$$ we have $\gamma_d = 1-\lambda(Q_d)$. The best-known upper and lower bounds on $\gamma_d$ for general $d$ are
\begin{equation}\label{eq:aks-bound} \frac{\log_2(d+2)}{2^{d+2}} \leq \gamma_d \leq \frac{1}{d+1},\end{equation}
and are due to Alon, Krech and Szab\'o \cite{aks}. The upper bound comes from taking $\mathcal{F} \subset \{0,1\}^n$ to consist of every $(d+1)^{\text th}$ layer of $\{0,1\}^n$; clearly, this set intersects the vertex-set of every $d$-dimensional subcube, and has density tending to $1/(d+1)$ as $n \to \infty$. In the case $d=2$, the result of Kostochka and Johnson-Entringer says that this construction is best-possible. The lower bound comes from observing that $g(d+2,d)\geq \log_2(d+2)$ and then partitioning $\{0,1\}^n$ into copies of $\{0,1\}^{d+2}$. (To see that $g(d+2,d) \geq \log_2(d+2)$, note that, if $\mathcal{F} \subset \{0,1\}^{d+2}$ intersects the vertex-set of every $d$-dimensional subcube, then it must separate the pairs of $[d+2]$, meaning that for any $1 \leq i < j \leq d+2$, there exists $x \in \mathcal{F}$ such that $x_i\neq x_j$; it is well-known, and easy to see, that if $\mathcal{F} \subset \{0,1\}^n$ separates the pairs of $[n]$ then $|\mathcal{F}| \geq \log_2 n$.) 

It is a folklore conjecture (see \cite{blm}) that $\gamma_d = 1/(d+1)$ for every integer $d \geq 2$, i.e.\ that the upper bound above is best-possible. In other words, to intersect the vertex-set of every $d$-dimensional subcube in $\{0,1\}^n$, one cannot asymptotically do better than to take every $(d+1)^{\text{th}}$ layer. The following result disproves this conjecture, showing that in fact $\gamma_d$ is exponentially small in $d$.

\begin{theorem}
\label{thm:cube} We have
    $$\gamma_d \leq c^d$$
    for all $d \in \mathbb{N}$, where $c <1$ is a constant. In fact, $\gamma_d < \frac{1}{d+1}$ for all $d \geq 8$.
\end{theorem}

A related question of Johnson and Talbot \cite[Question 13]{jt} is as follows: is it true that for any $\delta>0$ and integer $d \geq 2$, there exists $n_0 = n(d,\delta) \in \mathbb{N}$ such that for all $n \geq n_0$, every subset of $\{0,1\}^n$ of density at least $\delta$ must contain at least ${d \choose \lfloor d/2\rfloor}$ points of some $d$-dimensional subcube? It is easy to see that the answer to this question is `yes' for $d=2$ and $d=3$. Indeed, any subset of $\{0,1\}^n$ of density greater than $1/n$ must contain two points at Hamming distance 2, by averaging over Hamming spheres of radius 1, and two such points are clearly contained in a common 2-dimensional subcube. Moreover, any subset of $\{0,1\}^n$ of density greater than $2/n$ must contain three points of the form $x\Delta\{i\},x\Delta\{j\},x\Delta\{k\}$ for some $x \in \{0,1\}^n$ and distinct $i,j,k \in [n]$, again by averaging over Hamming spheres of radius 1, and three such points are clearly contained in a common 3-dimensional subcube.

However, one can use Theorem \ref{thm:main} to answer the Johnson-Talbot question in the negative for all integers $d \geq 4$, in the following quantitative form.

\begin{theorem}
    There exists a subset $\mathcal{F} \subset \{0,1\}^n$ such that every 4-dimensional subcube of $\{0,1\}^n$ contains at most five points of $\mathcal{F}$, with 
    $$\frac{|\mathcal{F}|}{2^n} = \frac{1}{3}\prod_{k=1}^{\infty}(1-2^{-k}) - o(1) \approx 0.097 - o(1).$$

    Moreover, for each $d \geq 5$, there exists a subset $\mathcal{F} \subset \{0,1\}^n$ such that every $d$-dimensional subcube of $\{0,1\}^n$ contains fewer than ${d \choose \lfloor d/2\rfloor}$ points of $\mathcal{F}$, with
    $$\frac{|\mathcal{F}|}{2^n} \geq c/\sqrt{d},$$
    where $c>0$ is an absolute constant.
    \label{thm:turan}
\end{theorem}

Johnson and Talbot \cite[Question 12]{jt} also asked if for any fixed family $\mathcal{S}$ with $\mathcal{S}\subset \{0,1\}^d$, there exist (for each $n \in \mathbb{N}$) sets $I_n \subset [n]$ such that the family $\mathcal{F}(I_n)=\{x \in \{0,1\}^n:\ \sum_{i=1}^{n}x_i \in I_n\}$ is $\mathcal{S}$-free, and 
$$\lim_{n \to \infty} \frac{|\mathcal{F}(I_n)|}{2^n} = \lambda(\mathcal{S}).$$

\noindent (Here, $\lambda(\mathcal S)$ denotes the asymptotically greatest density of a family $\mathcal F\subset\{0,1\}^n$ that is $\mathcal S$-free, in the sense that $\mathcal F$ does not contain any image of $\mathcal S$ under an isometric embedding of $Q_d$ into $Q_n$.)
In other words, is it true that an $\mathcal{S}$-free family of asymptotically maximum size can be obtained by taking an appropriate union of layers of the hypercube? Theorem \ref{thm:cube} gives a negative answer to this question, even for the special case of $\mathcal{S}=\{0,1\}^d = V(Q_d)$, for any $d \geq 8$: it is easy to see that a union of layers that is $Q_d$-free has asymptotic density at most $1-1/(d+1)$, whereas Theorem \ref{thm:cube} yields $Q_d$-free subsets of $\{0,1\}^n$ of asymptotic density greater than $1-1/(d+1)$ for each $d \geq 8$. 

\vspace{0.5cm}

Edge-Tur\'an problems in the hypercube have been perhaps more widely studied than vertex-Tur\'an problems, though in some ways the latter are more natural. A $\$ 100$ problem of Erd\H{o}s, still unsolved, asks for the maximum possible number of edges of a subgraph of $Q_n$ which contains no subgraph isomorphic to $C_4$. In general, if $H$ is a graph and $n \in \mathbb{N}$, we write $\ex(Q_n,H)$ for the maximum possible number of edges of a subgraph of $Q_n$ that contains no subgraph isomorphic to $H$. Erd\H{o}s in fact raised the problem of finding (or bounding) $\ex(Q_n,C_{2k})$ for all $k \in \mathbb{N}$. Alon, Krech and Szab\'o proposed a different generalisation of Erd\H{o}s' $C_4$-problem: that of determining (or bounding) $\ex(Q_n,Q_d)$ for each integer $d \geq 2$. Again, it is more convenient to consider the (equivalent) complementary problem: what is the minimal size $f(n,d)$ of a subgraph of $Q_n$ whose edge-set intersects the edge-set of every subgraph of $Q_n$ that is isomorphic to $Q_d$? An easy averaging argument shows that $f(n,d)/e(Q_n)$ is increasing in $n$, so the limit
$$\rho_d = \lim_{n \to \infty} f(n,d)/e(Q_n)$$
exists. Of course, a subgraph of $Q_n$ isomorphic to $Q_d$ must be induced by the vertex-set of a $d$-dimensional subcube of $Q_n$, so $\rho_d$ is a very close analogue of $\gamma_d$, the vertex version considered above. The best-known upper and lower bounds on $\rho_d$ for general $d$ are
$$\Omega\left(\frac{\log_2d}{2^{d}}\right) \leq \rho_d \leq O(1/d^2),$$
and are due to Alon, Krech and Szab\'o \cite{aks}. Theorem \ref{thm:cube} immediately yields an exponential upper bound on $\rho_d$, simply by taking the set of all edges of $Q_n$ that are incident to at least one vertex in a family $\mathcal{F} \subset \{0,1\}^n$ that intersects the vertex-set of every $d$-dimensional subcube.
\begin{theorem}
\label{thm:cube-edges} We have
    $$\rho_d \leq c^d$$
    for all $d \in \mathbb{N}$, where $c<1$ is a constant.
\end{theorem}

Finally, we turn to the conjecture of Bukh \cite{bukh} and of Griggs and Lu (see \cite{griggs}). They conjectured that, for any poset $\mathcal P$, the asymptotically greatest density of a subset of the $n$-dimensional Boolean lattice not containing a copy of $\mathcal P$ as a subposet is always attained by taking some union of layers.

Now, for the poset $\mathcal P$ being the $d$-dimensional Boolean lattice, the best collection of layers to take would be the middle $d$ layers (of the $n$-dimensional lattice). However, instead we may take the middle $d+1$ layers of the $n$-dimensional lattice, and replace the middle layer by a daisy-free family as given by Theorem \ref{thm:main}. For any $d \geq 4$, this set does not contain a copy of $\mathcal P$, and its size is asymptotically strictly greater than that of the middle $d$ layers.

\section{Daisies have positive Tur\'an density}
In this section we prove Theorems \ref{thm:main}, \ref{thm:larger} and \ref{theorem:special}. We start with Theorem \ref{thm:main}. 
The key idea will be to take (a blow-up of) a linear-algebraic construction. The construction itself will be for $n=2^r-1$. The reader will notice that it is because this value is more than quadratic in $r$ that the blow-up then preserves positive density.

\begin{proof}[Proof of Theorem \ref{thm:main}.] Fix $r\geq 2$. We first consider the case when $n=2^r-1$. We identify $[n]$ with $\mathbb{F}_2^r \setminus \{0\}$, the set of all nonzero vectors of length $r$ over the finite field $\mathbb{F}_2$. 

Take $\mathcal{F}$ to consist of all bases of $\mathbb{F}_2^r$ --- that is, all the linearly independent subsets of size $r$. We claim that $\mathcal F$ is $\mathcal D_r$-free. 

To show this, we must prove that any copy of $\mathcal{D}_r$ in $(\mathbb{F}_2^r \setminus \{0\})^{(r)}$ contains a linearly dependent set. Suppose for a contradiction that we have a copy of $\mathcal{D}_r$ in $(\mathbb{F}_2^r \setminus \{0\})^{(r)}$ whose $r$-sets are all linearly independent. Let $S = \{v_1,v_2,\ldots,v_{r-2}\}$ denote the stem of this daisy; then $S$ is certainly a linearly independent set. Let $W$ denote the subspace spanned by $S$, so that $W$ is an $(r-2)$-dimensional subspace of $\mathbb{F}_2^r$. Let $T = \{u_1,u_2,u_3,u_4\}$ denote the other four vertices of the daisy. Since $\dim(\mathbb{F}_2^r/W)=2$, we have $|\mathbb{F}_2^r/W|=4$. Hence, either there exists $i \in [4]$ such that $u_i \in W$ (in which case $\{u_i,u_j,v_1,v_2,\ldots,v_{r-2}\}$ is not linearly independent, for any $j \in [4]\setminus\{i\}$, a contradiction), or else there exist distinct $i,j \in [4]$ such that $u_i+W=u_j+W$ (in which case $\{u_i,u_j,v_1,v_2,\ldots,v_{r-2}\}$ is not linearly independent, again a contradiction). This shows that $\mathcal{F}$ is indeed $\mathcal{D}_r$-free, as claimed. 

The density of $\mathcal F$ is easy to calculate:
\begin{align*} \frac{|\mathcal{F}|}{{n \choose r}} & = \frac{\text{number of ordered bases of } \mathbb{F}_2^r}{n(n-1)(n-2)\ldots(n-r+1)}\\
& = \frac{(2^r-1)(2^r-2)(2^r-4)\cdots (2^r-2^{r-1})}{(2^r-1)(2^r-2)(2^r-3)\cdots(2^r-r)}\\
& > \frac{(2^r-1)(2^r-2)(2^r-4)\cdots (2^r-2^{r-1})}{2^{r^2}}\\
& = \prod_{k=1}^{r}(1-2^{-k})\\
& > \prod_{k=1}^{\infty} (1-2^{-k})\\
& \approx 0.29.
\end{align*}
We conclude that
$$\frac{\ex(2^r-1,\mathcal{D}_r)}{{2^r-1 \choose r}} > \prod_{k=1}^{r} (1-2^{-k}).$$

We now turn to general $n$, for which we will use a `blow-up' of the above construction. 
For $n$ a multiple of $2^{r}-1$, partition $[n]$ into $2^r-1$ classes of equal size, $C_1,\ldots,C_{2^r-1}$ say, and consider the family $\mathcal{G}$ consisting of all $r$-element subsets of $[n]$ of the form $\{x_1,\ldots,x_r\}$, where $x_i \in C_{j_i}$ for all $i \in [r]$, for $\{j_1,j_2,\ldots,j_r\}$ some element of $\mathcal{F}$, the above $\mathcal{D}_r$-free family on $[2^r-1]$. Clearly $\mathcal{G}$ is $\D_r$-free, since $\mathcal{F}$ is. We now estimate the density of $\mathcal{G}$. The probability that a uniformly random $r$-element subset of $[n]$ has at least two points in some $C_i$ is, by a union bound, at most
$$\frac{{r \choose 2}}{2^r-1}.$$
(To see this, note that if the uniformly random $r$-set $\{x_1,\ldots,x_r\}$ is chosen by choosing $x_1,\ldots,x_r$ at random in order without replacement, then for any $1\leq i < j \leq r$ and any $k \leq 2^r-1$, the probability that both $x_i$ and $x_j$ are in $C_k$ is $|C_k|(|C_k|-1)/(n(n-1)) < 1/(2^r-1)^2$.) Therefore,
$$\frac{|\mathcal{G}|}{{n \choose r}} \geq \left(1-\frac{{r \choose 2}}{2^r-1}\right)\frac{|\mathcal{F}|}{{2^r-1 \choose r}}>\left(1-\frac{{r \choose 2}}{2^r-1}\right)\prod_{k=1}^{r} (1-2^{-k}).$$
Hence 
$$\frac{\ex(n,\D_r)}{{n \choose r}} >\left(1-\frac{{r \choose 2}}{2^r-1}\right)\prod_{k=1}^{r} (1-2^{-k}).$$
Since this holds for arbitrarily large multiples $n$ of $2^r-1$, it follows that
$$\pi(\D_r) \geq \left(1-\frac{{r \choose 2}}{2^r-1}\right)\prod_{k=1}^{r} (1-2^{-k}).$$
Taking the limit of the above as $r \to \infty$ yields 
$$\lim_{r \to \infty}\pi(\mathcal{D}_r) \geq \prod_{k=1}^{\infty}(1-2^{-k}) \approx 0.29,$$
as required.
\end{proof}

The proof of Theorem \ref{thm:larger} is very similar. The daisy in Theorem \ref{thm:main} has $t-2=2$, and the above construction will generalise straightforwardly when $t-2$ is a prime power. Standard estimates on primes will then allow us to pass to the general case. 

\begin{proof}[Proof of Theorem \ref{thm:larger}.]
We start with the case when $t-2$ is a prime power, $q$ say. Let $n=q^r-1$ and identify $[n]$ with $\mathbb{F}_q^r \setminus \{0\}$, where $\mathbb{F}_q$ is the field of order $q$. Take $\mathcal{F}$ to consist of all bases of $\mathbb{F}_q^r$. 

We now claim that $\mathcal{F}$ is $\mathcal{D}_r(2,t)$-free. To show this, it suffices to prove that any copy of $\mathcal{D}_r(2,t)$ in $(\mathbb{F}_q^r \setminus \{0\})^{(r)}$ must contain a linearly dependent set. Suppose for a contradiction that we have a copy of $\mathcal{D}_r(2,t)$ in $(\mathbb{F}_q^r \setminus \{0\})^{(r)}$, whose $r$-sets are all linearly independent sets. Let $S = \{v_1,v_2,\ldots,v_{r-2}\}$ denote the stem of this daisy; then $S$ is a linearly independent set. Let $W$ denote the subspace spanned by $S$, so that $W$ is an $(r-2)$-dimensional subspace of $\mathbb{F}_q^r$. Let $T = \{u_1,u_2,\ldots,u_{q+2}\}$ denote the other $t=q+2$ vertices of the daisy. Then $\mathcal{S}=\{u_1+W,u_2+W,\ldots,u_{q+2}+W\}$ is a set of size $q+2$ in the two-dimensional vector space $\mathbb{F}_q^r/W$ (over $\mathbb{F}_q$), so $\mathcal{S}$ contains at least two points on the same line through the origin (noting that the $q+1$ distinct lines through the origin in $\mathbb{F}_q^2$ cover $\mathbb{F}_q^2$), so there exist $1 \leq i < j \leq q+2$ and $\lambda \in \mathbb{F}_q$ such that $u_i+W = \lambda(u_j+W)$. But then $\{u_i,u_j,v_1,v_2,\ldots,v_{r-2}\}$ is a linearly dependent set, a contradiction. This shows that $\mathcal{F}$ is indeed $\mathcal{D}_r(2,t)$-free, as claimed.

As above, we have:
\begin{align*} \frac{|\mathcal{F}|}{{n \choose r}} & = \frac{\text{number of ordered bases of } \mathbb{F}_q^r}{n(n-1)(n-2)\ldots(n-r+1)}\\
& = \frac{(q^r-1)(q^r-q)(q^r-q^2)\cdots (q^r-q^{r-1})}{(q^r-1)(q^r-2)(q^r-3)\cdots(q^r-r)}\\
& > \frac{(q^r-1)(q^r-q)(q^r-q^2)\cdots (q^r-q^{r-1})}{q^{r^2}}\\
& = \prod_{k=1}^{r}(1-q^{-k})\\
& > \prod_{k=1}^{\infty} (1-q^{-k})\\
& = 1-1/q-O(1/q^2).
\end{align*}
Hence
$$\frac{\ex(q^r-1,\mathcal{D}_r(2,t))}{{q^r-1 \choose r}} > \prod_{t=1}^{r} (1-q^{-t}).$$

Exactly the same blow-up construction as in the proof of Theorem \ref{thm:main} yields
$$\frac{\ex(n,\D_r(2,t))}{{n \choose r}} >\left(1-\frac{{r \choose 2}}{q^r-1}\right)\prod_{k=1}^{r} (1-q^{-k})$$
for $n$ an arbitrarily large multiple of $q^r-1$. It follows that
$$\pi(\D_r(2,t)) \geq \left(1-\frac{{r \choose 2}}{q^r-1}\right)\prod_{k=1}^{r} (1-q^{-k}).$$
Taking the limit of the above as $r \to \infty$ yields 
\begin{align}
\label{eq:useful-fact}\lim_{r \to \infty}\pi(\mathcal{D}_r(2,t))& \geq \prod_{k=1}^{\infty}(1-q^{-k})\\
&=1-1/q-O(1/q^2). \nonumber
\end{align}

Finally, for general values of $t$, one may use the fact (from Baker, Harman and Pintz \cite{bhp}) that for any sufficiently large $x>0$ there is a prime between $x-x^{0.525}$ and $x$. By choosing the greatest prime $q$ such that $q\leq t-2$, we obtain from the above that
$$\lim_{r \to \infty}\pi(\mathcal{D}_r(2,t)) \geq 1-1/t-O(1/t^{1.475})$$
for all $t$. 
\end{proof}

We now turn to the proof of Theorem \ref{theorem:special}. In fact, we prove the following more general result, which we include because it 
gives the best bounds that we
can produce for general $(r,s,t)$-daisies. Note that
Theorem \ref{theorem:special} is precisely the case of this when $s=m=k$ and $q=2$.

\begin{theorem}
\label{thm:even-larger}
    Let $q$ be a prime power, $m$ a positive integer and $s$ even, and put $t= 
    \lfloor sq^{\frac{2m}{s}+1}\rfloor+1$. Then
    $$\lim_{r \to \infty} \pi(\mathcal{D}_{r}(s,t)) \geq 1-q^{-(m+1)}-O(q^{-(m+2)}).$$

    
\end{theorem}

[Here and elsewhere in the paper the implied constant in the `O' notation is an absolute constant.]
\vspace{0.5cm}

We first need some preliminaries. For $s,d \in \mathbb{N}$ and $q$ a prime power, we say that a family of vectors $\mathcal{X} \subset \mathbb{F}_q^d$ is {\em $s$-wise independent} if any $s$ of the vectors in $\mathcal{X}$ are linearly independent. We will make
use of the following lemma, due to Earnest \cite{earnest}. We reproduce Earnest's short proof, for the reader's convenience.

\begin{lemma}
\label{lemma:k-wise}
    Let $q$ be a prime power, let $s$ be an even positive integer, and let $\mathcal{X} \subset \mathbb{F}_q^d$ be $s$-wise independent. Then
    $$|\mathcal{X}| \leq sq^{\frac{2d}{s}-1}.$$
\end{lemma}
\begin{proof}
    We first show that 
    $$(q-1)^{s/2}{|\mathcal{X}| \choose s/2} \leq q^d.$$
    Indeed, write $l=s/2$. Then the $(q-1)^{l}{|\mathcal{X}| \choose l}$ sums of the form
    $$c_1 v_1 + c_2 v_2 + \ldots + c_l v_l,$$
    for $\{v_1,v_2,\ldots,v_l\}$ an $l$-element subset of $\mathcal{X}$ and $c_1,c_2,\ldots,c_l \in \mathbb{F}_q\setminus \{0\}$, are distinct elements of $\mathbb{F}_q^d$, otherwise $\mathcal{X}$ would contain a linearly dependent subset of size $2l=s$. This establishes the claim.

    If we now use the inequality $${|\mathcal{X}| \choose s/2} \geq \left(\frac{|\mathcal{X}|}{s/2}\right)^{s/2},$$
we obtain
    $$|\mathcal{X}| \leq \frac{sq}{2(q-1)} q^{\frac{2d}{s}-1}\leq sq^{\frac{2d}{s}-1}.$$
\end{proof}

\begin{proof}[Proof of Theorem \ref{thm:even-larger}.]
Let $m \in \mathbb{N}$, let $n=q^{m+r}-1$, identify $[n]$ with $\mathbb{F}_q^{m+r}\setminus \{0\}$, and take $\mathcal{F}$ to consist of all linearly independent sets of $r$ vectors in $\mathbb{F}_q^{m+r}$. We then have 
\begin{align*} \frac{|\mathcal{F}|}{{n \choose r}} & = \frac{\text{number of ordered linearly independent \textit{r}-element subsets of } \mathbb{F}_q^{m+r}}{n(n-1)(n-2)\ldots(n-r+1)}\\
& = \frac{(q^{m+r}-1)(q^{m+r}-q)(q^{m+r}-q^2)\cdots (q^{m+r}-q^{r-1})}{(q^{m+r}-1)(q^{m+r}-2)(q^{m+r}-3)\cdots(q^{m+r}-r)}\\
& > \frac{(q^{m+r}-1)(q^{m+r}-q)(q^{m+r}-q^2)\cdots (q^{m+r}-q^{r-1})}{q^{r(m+r)}}\\
& = \prod_{k=1}^{r}(1-q^{-(m+k)})\\
& > \prod_{k=1}^{\infty} (1-q^{-(m+k)})\\
& = 1-1/q^{m+1}-O(1/q^{m+2}).
\end{align*}

Let $t = \lfloor sq^{\frac{2m}{s}+1}\rfloor+1$. We claim that $\mathcal{F}$ is $\mathcal{D}_r(s,t)$-free. To show this, it suffices to prove that any copy of $\mathcal{D}_r(s,t)$ in $(\mathbb{F}_q^r \setminus \{0\})^{(r)}$ must contain a linearly dependent set. Suppose for a contradiction that we have a copy of $\mathcal{D}_r(s,t)$ in $(\mathbb{F}_q^{m+r} \setminus \{0\})^{(r)}$, whose $r$-sets are all linearly independent. Let $S = \{v_1,v_2,\ldots,v_{r-s}\}$ denote the stem of this daisy; then $S$ is a linearly independent set. Let $W$ denote the subspace spanned by $S$, so that $W$ is an $(r-s)$-dimensional subspace of $\mathbb{F}_q^{m+r}$. Let $T = \{u_1,u_2,\ldots,u_{t}\}$ denote the other $t$ vertices of the daisy. Then $\mathcal{S}=\{u_1+W,u_2+W,\ldots,u_{t}+W\}$ is a set of size $t$ in the $(m+s)$-dimensional vector space $\mathbb{F}_q^{m+r}/W$ (over $\mathbb{F}_q$), so by Lemma \ref{lemma:k-wise}, applied with $d=m+s$, we have
that $\mathcal{S}$ contains a linearly dependent $s$-set, $\{u_{i_1}+W,u_{i_2}+W,\ldots,u_{i_s}+W\}$ say. But then $\{v_1,v_2,\ldots,v_{r-s},u_{i_1},u_{i_2},\ldots,u_{i_s}\}$ is a linearly dependent $r$-set, a contradiction.

The statement on Tur\'an densities now follows by the same blow-up argument as in the proofs of Theorem \ref{thm:main} and Theorem \ref{thm:larger}.
\end{proof}

\section{Proofs of Theorems \ref{thm:cube} and \ref{thm:turan}}
We first prove Theorem \ref{thm:cube}. Our main tool will be Theorem \ref{theorem:special}. We remark that just to disprove the statement that $\gamma_d=1/(d+1)$, it would be sufficient to use Theorem \ref{thm:larger}, but this would only show that $\gamma_d=O(1/d^2)$; to obtain the exponential bound on $\gamma_d$, we need Theorem \ref{theorem:special}. We also remark that the proof below that $\gamma_d<1/(d+1)$ for all $d\geq8$ is included just for precision; the reader who is only interested in large values of $d$ may omit this technically-optimised part of the proof. 

\begin{proof}[Proof of Theorem \ref{thm:cube}.]
    Note that, since $\pi(\D_r(k,8k+1))$ is decreasing in $r$ and $\ex(n,\D_r(k,8k+1))/{n \choose r}$ is decreasing in $n$, it follows from Theorem \ref{theorem:special} that $\ex(n,\D_r(k,8k+1))/{n \choose r} = 1-O(2^{-k})$ for all $r$ and all $n \geq r$.  
    
    In proving the theorem, by adjusting the value of $c$ if necessary, we may clearly assume that $d \geq d_0$, for any absolute constant $d_0 \in \mathbb{N}$.
    
    Given an integer $d \geq 17$, we construct a family $\mathcal{F} \subset \{0,1\}^n$ intersecting the vertex-set of every copy of $Q_d$ as follows. Identify $\{0,1\}^n$ with $\mathcal{P}([n])$ in the usual way. Let $k$ be the maximum even positive integer such that $8k+1 \leq d$. For every integer $r$ with $1 \leq r \leq n$, let $\mathcal{F}_r \subset [n]^{(r)}$ with $|\mathcal{F}_r|/{n \choose r} \leq O(2^{-k})$ and such that $\mathcal{F}_r$ intersects every copy of $\mathcal{D}_r(k,8k+1)$ in $[n]^{(r)}$. (Such exists, by the remark at the beginning of the proof.)
    Let $\mathcal{F} = \cup_{r=0}^{n}\mathcal{F}_r$. Then $\mathcal{F}$ intersects every copy of $\mathcal{D}_r(k,8k+1)$ in $\mathcal{P}([n])$, for all $r$. Since the vertex-set of every copy of $Q_d$ in $\mathcal{P}([n])$ contains some copy of $\mathcal{D}_r(k,8k+1)$, it follows that $\mathcal{F}$ intersects the vertex-set of every copy of $Q_d$ in $\mathcal{P}([n])$. Clearly, we have
    $$\frac{|\mathcal{F}|}{2^n} = O(2^{-k})= O(2^{-d/8}),$$
    and therefore
    $$\gamma_d = O(2^{-d/8}),$$
    proving the first part of the theorem.\\

    To prove that $\gamma_d < 1/(d+1)$ for all $d \geq 8$, we use a fact from the proof of Theorem \ref{thm:larger}, and a more careful insertion of sparse sets intersecting every daisy, into a sparse set of layers. We note that, since $\pi(\D_r(2,t))$ is decreasing in $r$ and $\ex(n,\D_r(2,t))/{n \choose r}$ is 
    decreasing in $n$, it follows from (\ref{eq:useful-fact}) (in the proof of Theorem \ref{thm:larger}) that $\ex(n,\D_r(2,q+2))/{n \choose r} \geq \prod_{k=1}^{\infty}(1-q^{-k})$ for all $r$, all $n$
    and all prime powers $q$.
    
    Let $t \leq \lfloor d/2 \rfloor +1$ be maximal such that $t-2$ is a prime power, $q$ say. View $\{0,1\}^n$ as $\mathcal{P}([n])$, the power-set of $[n]$. For each $r\leq n$ such that $r$ is a multiple of $\lceil d/2\rceil$, let $\mathcal{F}_r$ be an $r$-uniform hypergraph on the vertex-set $[n]$ that has density at most
    $$1-\prod_{k=1}^{\infty}(1-q^{-k})$$
    and intersects every copy of $\D_r(2,t)$.
    
    Now let $\mathcal{F} = \cup_{r=0}^{n}\mathcal{F}_r$. Consider the vertex-set $\mathcal{V}$ of any copy of $Q_d$ in $Q_n$; $\mathcal{V}$ intersects $d+1$ layers of $\mathcal{P}([n])$, say the layers $r_0,\ldots,r_0+d$, and for each layer $r_0 + 2 \leq i \leq r_0+\lceil d/2 \rceil+1$, the intersection of $\mathcal{V}$ with the $i^{\text{th}}$ layer contains a copy of $\D_r(2,t)$. One of these $\lceil d/2 \rceil$ values of $i$ must be a multiple of $\lceil d/2\rceil$, so $\mathcal{V}$ must intersect $\mathcal{F}$. Therefore $\mathcal{F}$ intersects the vertex-set of every $d$-dimensional subcube of $\{0,1\}^d$, as needed. The family  $\mathcal{F}$ has density at most
    \begin{align*} \frac{1}{\lceil d/2\rceil}\left(1-\prod_{k=1}^{\infty}(1-q^{-k})\right)+o_{n \to \infty}(1),\end{align*}
    so
    $$\frac{g(n,d)}{2^n} \leq \frac{1}{\lceil d/2\rceil}\left(1-\prod_{k=1}^{\infty}(1-q^{-k})\right)+o_{n \to \infty}(1).$$
    Taking the limit of the above as $n \to \infty$ yields
    \begin{equation}\label{eq:master}\gamma_d \leq \frac{1}{\lceil d/2\rceil}\left(1-\prod_{k=1}^{\infty}(1-q^{-k})\right)\end{equation}
    for all $d \geq 6$.

  It suffices (by the inequality (\ref{eq:master})) to check that, for each $d \geq 8$, if $q(d)$ is the largest prime power $q$ such that $q \leq \lfloor d/2\rfloor -1$ then
    \begin{equation}
    \label{eq:suff}
    \frac{1}{\lceil d/2\rceil}\left(1-\prod_{k=1}^{\infty}(1-q(d)^{-k})\right) < \frac{1}{d+1}.\end{equation}
    Write $\beta_q = \prod_{k=1}^{\infty}(1-q^{-k})$ for each prime power $q$; clearly, $\beta_q$ is an increasing function of $q$. For each odd $d \geq 9$, (\ref{eq:suff}) is equivalent to
    $$\beta_{q(d)} > 1/2,$$
and this holds, since for all $d \geq 9$ we have $\beta_{q(d)} \geq \beta_{q(9)} = \beta_3 \approx 0.560 > 1/2$. For each even $d \geq 8$, (\ref{eq:suff}) is equivalent to 
$$\beta_{q(d)} > \tfrac{1}{2}+\tfrac{1}{2d+2},$$
which holds since for all $d \geq 8$ we have
$$\beta_{q(d)} \geq \beta_{q(8)} = \beta_3 \approx 0.560 > \tfrac{5}{9}\geq \tfrac{1}{2}+\tfrac{1}{2d+2},$$
noting that the right-hand side is a decreasing function of $d$.
\end{proof}
    
We now turn to the proof of Theorem \ref{thm:turan}.
\begin{proof}[Proof of Theorem \ref{thm:turan}.]
    We first prove the first part of the theorem. Note that it follows from Theorem \ref{thm:main} that $\ex(n,\D_r)/{n \choose r} \geq \prod_{k=1}^{\infty}(1-2^{-k}) \approx 0.29$ for all $n$ and $r$.  
    
    View $\{0,1\}^n$ as $\mathcal{P}([n])$. For each $r \leq n$ such that $r$ is a multiple of $3$, let $\mathcal{F}_r$ be an $r$-uniform hypergraph on the vertex-set $[n]$ that has density at least
    $$\prod_{k=1}^{\infty}(1-2^{-k}) \approx 0.29$$
    and contains no copy of $\D_r$. Now let $\mathcal{F}=\cup_{r=0}^{n}\mathcal{F}_r$. We clearly have
    $$\frac{|\mathcal{F}|}{2^n} = \tfrac{1}{3}\prod_{k=1}^{\infty}(1-2^{-t}) - o(1) \approx 0.097 - o(1).$$
    We claim that $\mathcal{F}$ contains at most five points of any 4-dimensional subcube. Indeed, let $\mathcal{V}$ denote the vertex-set of a 4-dimensional subcube of $\{0,1\}^n$; then $\mathcal{V}$ intersects five layers of $\mathcal{P}([n])$, say the layers $r_0,r_0+1,r_0+2,r_0+3$ and $r_0+4$. If $r_0$ is congruent to 0 modulo 3, then there are just five vertices of $\mathcal{V}$ at layers of $\mathcal{P}([n])$ congruent to 0 modulo 3 (namely, those in layers $r_0$ or $r_0+3$), so the claim trivially holds. For the same reason, the claim trivially holds if $r_0$ is congruent to 2 modulo 3. If $r_0$ is congruent to 1 modulo 3, then there are just six vertices of $\mathcal{V}$ at layers $r$ of $\mathcal{P}([n])$ congruent to 0 modulo 3, namely those in the $(r_0+2)^{\text{th}}$ layer, and $\mathcal{F}$ cannot contain all six of them, since they form a copy of $\mathcal{D}_{r_0+2}$. This proves the claim, completing the proof of the first part of the theorem.
\vspace{0.5cm}

    We now prove the second part of the theorem. Let $C>0$ be a (large) absolute constant to be chosen later. View $\{0,1\}^n$ as $\mathcal{P}([n])$. For each $r \leq n$ such that $r$ is a multiple of $\lceil C\sqrt{d}\rceil$, let $\mathcal{F}_r$ be an $r$-uniform hypergraph on the vertex-set $[n]$ that has density at least
    $$\prod_{k=1}^{\infty}(1-2^{-k}) \approx 0.29$$
    and contains no copy of $\D_r$. Now let $\mathcal{F}=\cup_{r=0}^{n}\mathcal{F}_r$. We clearly have
    $$\frac{|\mathcal{F}|}{2^n} = \tfrac{1}{\lceil C\sqrt{d}\rceil}\prod_{k=1}^{\infty}(1-2^{-k}) + o_{n \to \infty}(1) \geq c/\sqrt{d},$$
    for some sufficiently small absolute constant $c>0$.
    
    We claim that, provided $C$ is sufficiently large, any $d$-dimensional subcube of $\{0,1\}^n$ contains less than ${d \choose \lfloor d/2\rfloor}$ points of $\mathcal{F}$. To see this, let $\mathcal{V}$ denote the vertex-set of a $d$-dimensional subcube of $\{0,1\}^n$; then $\mathcal{V}$ intersects $d+1$ layers of $\mathcal{P}([n])$, say the layers $r_0,r_0+1,\ldots,r_0+d-1$ and $r_0+d$. Note that, among these layers, there is at most one layer $i$ such that $i$ is a multiple of $\lceil C \sqrt{d}\rceil$ and $|i-(r_0+d/2)| < C\sqrt{d}/2$. For such a layer $i$, since $\mathcal{F} \cap [n]^{(i)}$ is $\mathcal{D}_i$-free, by averaging we have $|\mathcal{V} \cap \mathcal{F} \cap [n]^{(i)}| \leq \tfrac{5}{6}{d \choose i-r_0} \leq \tfrac{5}{6}{d \choose \lfloor d/2\rfloor}$. Let $\mathcal{U}$ denote the union of all the layers $j$ of $\mathcal{P}([n])$ that are multiples of $\lceil C \sqrt{d}\rceil$ but satisfy $|j-(r_0+d/2)| \geq C\sqrt{d}/2$; then since $|\mathcal{V} \cap [n]^{(j)}| < |\mathcal{V} \cap [n]^{(k)}|$ if $k$ is closer to $r_0+d/2$ than $j$ is, we have $|\mathcal{V} \cap \mathcal{U}| < 2^d/(\lfloor C\sqrt{d}/2\rfloor-1) < \tfrac{1}{6}{d \choose \lfloor d/2\rfloor}$ provided $C$ is sufficiently large. Hence, overall we have
    $$|\mathcal{V} \cap \mathcal{F}| < \tfrac{5}{6}{d \choose \lfloor d/2\rfloor}+\tfrac{1}{6}{d \choose \lfloor d/2\rfloor}={d \choose \lfloor d/2\rfloor},$$
as claimed.
\end{proof}
We remark that the second part of the previous theorem is clearly best-possible up to the value of the absolute constant $c$ since, by averaging, any subset $\mathcal{F}\subset \{0,1\}^n$ of density at least ${d \choose \lfloor d/2\rfloor}/2^d$ contains at least ${d \choose \lfloor d/2\rfloor}$ points of some $d$-dimensional subcube.

\section{Concluding remarks}
It would be of interest to determine the right growth
speed of $\gamma_d$, in other words to determine 
$$\lim_{d \to \infty} \frac{\log_2(1/ \gamma_d)}{d},$$
if this limit exists. The quantity  
$\frac{\log_2(1/ \gamma_d)}{d}$ is asymptotically at least $1/8$ (by the proof of Theorem \ref{thm:cube}), and
of course it is always at most 1.

Another open problem, first raised in \cite{blm}, is to determine exactly the Tur\'an density of $\mathcal{D}_3$. It is conjectured in \cite{blm} that $\pi(\mathcal{D}_3)=1/2$. In \cite{blm} there is a construction attaining this asymptotic density, by taking the complement of the Fano plane, blowing up and iterating. 

It would also be interesting to understand what happens for daisies when $s=3$, in particular to determine whether or not
\begin{equation}\label{eq:conj} \lim_{r \to \infty}\pi(\mathcal{D}_r(3,t)) =1-O(1/t^2).\end{equation}
We note that (\ref{eq:conj}) would be best-possible, up to the implicit constant. Indeed, a result of de Caen \cite{caen} states that $\pi(K_t^{(3)}) \leq 1-1/{t-1 \choose 2}$ for all $t \geq 3$. Therefore, for any $\epsilon >0$, any $r \geq 3$ and any $r$-uniform hypergraph $\mathcal{F}$ on the vertex-set $[n]$ with density at least $1-1/{t-1\choose 2}+\epsilon$, if $n$ is sufficiently large depending on $r, t$ and $\epsilon$ then there exists an $(r-3)$-set $S$ such that the link hypergraph $$\{\{i,j,k\} \in ([n]\setminus S)^{(3)}:\ S \cup \{i,j,k\} \in \mathcal{F}\}$$ has density at least $1-1/{t-1 \choose 2}+\epsilon$. So by de Caen's theorem this link hypergraph must contain a $K_t^{(3)}$, and so $\mathcal{F}$ must contain a copy of $\D_r(3,t)$. This shows that $\pi(\D_r(3,t)) \leq 1-1/{t-1 \choose 2}$ for all $r,t \geq 3$.

Finally, in connection with forbidden subposet problems, observe that our construction does not disprove the longstanding conjecture that for the diamond poset, or equivalently the two-dimensional Boolean lattice, the greatest asymptotic density of a diamond-free family of points in the $n$-dimensional Boolean lattice is achieved by taking the union of two middle layers.\\

\textbf{Note.} We are very grateful to Noga Alon for pointing out that if, instead of using the bound of
Earnest for $s$-wise independent families, one uses
the bound of Plotkin (M.~Plotkin, Binary codes with specified minimum distance, {\em IRE Transactions on Information Theory} 6 (1960), 445-450), then one obtains 
that the quantity $\frac{\log_2(1/ \gamma_d)}{d}$ above
is asymptotically at least $1/2$ -- in other words, that we may take the constant $c$ in Theorem \ref{thm:cube} to
be (asymptotically) $1/\sqrt{2}$. We are also grateful to Robert Johnson, and independently Noga Alon, for the following attractive observations. If one views the $r$-daisy as an $(r-2)$-set
connected to the six 2-sets that form the graph $K_4$,
then the proof of Theorem \ref{thm:main} goes through verbatim if we replace the graph $G=K_4$ by any graph $G$ having chromatic number at least 4. In contrast, if $G$ has chromatic number at most 3 then the corresponding Tur\'an densities have to tend to zero. Indeed, this is clear if
$G$ is a triangle (as then our family of $r$-sets would have to contain at most two $r$-sets in any given $(r+1)$-set), and in general it follows by `blowing up' this argument, via some averaging and known facts about
Tur\'an densities for 3-partite 3-graphs.\\

\textbf{Acknowledgement.} We are very grateful to Boris Bukh for pointing out to us the connection between our results and the conjecture of his and of Griggs and Lu on forbidden subposets.

\Addresses
\end{document}